\documentclass[reqno,11pt]{amsart}
\usepackage[margin=1in]{geometry}

\usepackage{amsthm, amsmath, amssymb, bm}
\usepackage{tikz}

\usepackage[utf8]{inputenc}
\usepackage[T1]{fontenc}

\usepackage[textsize=scriptsize,backgroundcolor=orange!5]{todonotes}

\usepackage[hidelinks]{hyperref}
\usepackage{url}
\usepackage{comment}
\usepackage{mathtools}
\usepackage{thmtools}
\usepackage{thm-restate}
\usepackage{float}
\usepackage{bbm}

\usepackage[noabbrev,capitalize]{cleveref}
\crefname{equation}{}{}

\usepackage{xcolor,graphics}

\usepackage[backend=biber,style=numeric]{biblatex}

\numberwithin{equation}{section}

\newtheorem{theorem}{Theorem}[section]
\newtheorem{proposition}[theorem]{Proposition}
\newtheorem{lemma}[theorem]{Lemma}

\newtheorem{corollary}[theorem]{Corollary}

\newtheorem*{question*}{Question} \Crefname{question}{Question}{Questions}

\theoremstyle{definition}
\newtheorem{definition}[theorem]{Definition}
\newtheorem{remark}[theorem]{Remark}

\newtheorem*{example*}{Example}

\newtheorem{hypothesis}[theorem]{Hypothesis}

\newcommand{\nc}{\newcommand}
\nc{\on}{\operatorname}

\newcommand{\one}{\mathbbm{1}}

\newcommand\numberthis{\addtocounter{equation}{1}\tag{\theequation}}

\nc{\fp}{\mathfrak{p}}
\nc{\fq}{\mathfrak{q}}
\nc{\fa}{\mathfrak{a}}
\nc{\fb}{\mathfrak{b}}
\nc{\fc}{\mathfrak{c}}
\nc{\fd}{\mathfrak{d}}
\nc{\fe}{\mathfrak{e}}
\nc{\ff}{\mathfrak{f}}

\nc{\cA}{\mathcal{A}}
\nc{\cP}{\mathcal{P}}
\nc{\cL}{\mathcal{L}}
\nc{\cR}{\mathcal{R}}
\nc{\cI}{\mathcal{I}}
\nc{\mc}{\mathcal}

\newcommand{\QQ}{\mathbb{Q}}
\newcommand{\RR}{\mathbb{R}}

\newcommand{\NN}{\mathbb{N}}
\newcommand{\ZZ}{\mathbb{Z}}

\newcommand{\Gal}{\operatorname{Gal}}
\newcommand{\eps}{\varepsilon}

\title{Primes with Beatty and Chebotarev conditions}

\author[Ji]{Caleb Ji}
\address{Washington University in St. Louis, St. Louis, MO 63130}
\email{caleb.ji@wustl.edu}

\author[Kazdan]{Joshua Kazdan}
\address{Stanford University, Stanford, CA 94305, USA}
\email{jkazdan@stanford.edu}

\author[McDonald]{Vaughan McDonald}
\address{Harvard University, Cambridge, MA 02138, USA}
\email{vmcdonald@college.harvard.edu}

\date{\today}

\allowdisplaybreaks

\begin{document}

\maketitle

\begin{abstract}
    We study the prime numbers that lie in Beatty sequences of the form $\lfloor \alpha n + \beta \rfloor$ and have prescribed algebraic splitting conditions. We prove that the density of primes in both a fixed Beatty sequence with $\alpha$ of finite type and a Chebotarev class of some Galois extension is precisely the product of the densities $\alpha^{-1}\cdot\frac{|C|}{|G|}$.  Moreover, we show that the primes in the intersection of these sets satisfy a Bombieri--Vinogradov type theorem.  This allows us to prove the existence of bounded gaps for such primes. As a final application, we prove a common generalization of the aforementioned bounded gaps result and the Green--Tao theorem.
\end{abstract}

\section{Introduction and Statement of Results}
Primes in a fixed arithmetic progression $\{a + qn\}_{n = 1}^{\infty}$ constitute a key object of interest in classical analytic number theory. The prime number theorem for arithmetic progressions computes the density of such primes.  We will discuss two natural generalizations of arithmetic progressions, one via an algebraic property and the other via an analytic one, and study primes constrained simultaneously by these conditions.

One way to generalize arithmetic progressions is by replacing the common integer difference $q$ and the constant term $a$ with arbitrary real constants $\alpha$ and $\beta$ respectively. We then consider the sequence $\{\lfloor\alpha n + \beta\rfloor\}$, which is known as a \emph{Beatty sequence}.
In an earlier paper, Banks and Shparlinski \cite{BS09} examine the density of primes that lie in certain Beatty sequences.  Another way to generalize arithmetic progressions is to note that reducing a prime$\pmod q$ is equivalent to taking its Chebotarev class in the cyclotomic field $\QQ(\zeta_q)$. Generalizing to any Galois extension $L/\QQ$, the Chebotarev density theorem computes the proportion of primes $p$ not dividing $\Delta_L$ that are contained in a given Chebotarev class.  

This paper studies the  primes $p$ that satisfy both conditions. Roughly speaking, our first theorem will show that within the primes, the property of lying in a Beatty sequence and the property of having a fixed Artin symbol in a Galois extension $L/\QQ$ are independent.  We first recall some relevant notions.
The natural density of a class of primes $\mathcal{S}\subset \mathbb{P}$ is given by 

\[ 
\delta(\mc{S}) \coloneqq \lim_{X\rightarrow\infty}\dfrac{\# \{ p\le X : p\in \mathcal{S}\}}{\#\{ p\le X : p\in \mathbb{P} \}},
\]
where $\mathbb{P}$ denotes the set of all primes.  Note that this limit does not exist for all sets $\mathcal{S}$. Given a Galois extension $L/\QQ$ of degree $d$ with Galois group $G$ and a prime $p \nmid \Delta_L$, let the \emph{Artin symbol} $\Big[\frac{L/\QQ}{p}\Big]$ be the conjugacy class of $\on{Frob}_{\mathfrak{p}}$ for any prime $\mathfrak{p} \subseteq \mc{O}_L$ lying above $(p)$.  
Let $C\subset G$ be a conjugacy class, and let $\cP_C$ be the set of primes $p\nmid\Delta_L$ such that $\Big[\frac{L/\QQ}{p}\Big]=C$.
The Chebotarev density theorem states that
\[\delta(\mc{P}_C)=\dfrac{|C|}{|G|}.\]


In essence, the probability a randomly chosen prime number has a fixed Artin symbol $C$ is $|C|/|G|$. Throughout this paper, we call set of primes with Artin symbols in a fixed subset of conjugacy classes in a given Galois extension $L/\QQ$ a \emph{Chebotarev set}. For the sake of this paper, however, all Chebotarev sets will contain primes with a fixed Artin symbol.

On the other hand, we also consider Beatty sequences of the form
\[
\mathcal{B}_{\alpha,\beta} =  \{\lfloor \alpha n + \beta \rfloor: n \in \NN\},
\] 
where $\alpha$ is an irrational number and $\beta$ is an arbitrary real number. We will require conditions on the \emph{type} $\tau$ of $\alpha$, which is given by
\[\tau = \sup \{ \rho : \liminf_{\substack{ n\rightarrow \infty \\ n\in \mathbb{N}}} n^\rho \lVert \alpha n \rVert=0\},\]
where $\lVert \cdot \rVert$ denotes the distance to the nearest integer.  Our results concerning primes in Beatty sequences rely on the condition that $\alpha$ is of finite type.  Most real numbers of interest have finite type, including $\pi$, $e$, and every algebraic number.  Let $\cP_{\alpha, \beta}$ be the set of primes appearing in the Beatty sequence $\lfloor \alpha n + \beta\rfloor$.
Banks and Shparlinski \cite{BS09} proved that if $\alpha$ has finite type, then  
\[
\delta(\cP_{\alpha, \beta}) = \dfrac{1}{\alpha},
\]
consistent with the observation that such a Beatty sequence will contain roughly $1/\alpha$ of all the natural numbers. In fact, they also proved that arithmetic progressions and Beatty sequences are independent: the density of primes in a fixed arithmetic progression and Beatty sequence is the product $\frac{1}{\varphi(q)\alpha}$.

We prove a common generalization of the results of Banks and Shparlinski and the Chebotarev density theorem. To be more precise, we wish to determine the density of the set 
\begin{equation}
\label{eq: Pcab}
    \cP_{C,\alpha,\beta} \coloneqq \cP_{\alpha,\beta} \cap \mathcal{P}_C.
\end{equation}

\begin{theorem}
\label{thm:chebeatty}
Let $L/\QQ$ be a finite Galois extension and $C\subset G=\Gal(L/\QQ)$ be a conjugacy class of its Galois group.  Suppose $\alpha$ is positive, irrational, and of finite type and $\beta$ is any real number. Then the density of the set $\cP_{C,\alpha,\beta}$ in the primes is
\begin{equation}
    \delta(\mathcal{P}_{C,\alpha,\beta}) = \delta(\cP_{\alpha, \beta}) \cdot \delta(\mathcal{P}_C) = \frac{1}{\alpha}\cdot \frac{|C|}{|G|}.
\end{equation}
\end{theorem}

The following example gives an explicit computational suggestion of our results.

\begin{example*} Consider the Beatty sequence $\mc{B}_{\pi,0}=\lfloor \pi n\rfloor$ and let $L$ be the splitting field of $x^3-2$ over $\QQ$. $L/\QQ$ is a Galois extension with $\on{Gal}(L/\QQ) \cong S_3$ and discriminant $\Delta_L = -2^4\cdot 3^7$.  For $g\in S_3$, let $C_{(g)}$ denote the conjugacy class of $g$. Then every prime $2,3$ has Artin symbol equal to one of $C_{(e)}, C_{(12)}, C_{(123)}$, with densities $\frac{1}{6}, \frac{1}{2}, \frac{1}{3}$, respectively. \cref{thm:chebeatty} shows the densities of the intersections with $\lfloor\pi n \rfloor$ are $\frac{1}{6 \pi}, \frac{1}{2 \pi}, \frac{1}{3 \pi}$ respectively.   Table 1 lists the proportion of the first $n$ primes (not including $2, 3$) of the form $\lfloor\pi n\rfloor$, with a certain conjugacy class of $\Gal(K/\QQ)$, and with both of these properties. 
Define
\[
\delta(S,X) = \frac{\#\{p \le X: p \in S\}}{\#\{p \le X\}}.
\]

\begin{table}[H]
\label{table: exampletable}
\begin{tabular}{|l|l|l|l|l|l|l|l|}
\hline
$X$ & $\delta(\mc{P}_{\pi},X)$ & $\delta(\mc{P}_{C_{e}},X)$ & $\delta(\cP_{C_{(12)}},X)$ & $\delta(\cP_{C_{(123)}},X)$ & $\delta(\cP_{C_{(e)},\pi},X)$ & $\delta(\cP_{C_{(12)},\pi},X)$ & $\delta(\cP_{C_{(123)},\pi},X)$ \\ \hline
100 & 0.35 & 0.15 & 0.52 & 0.33 & 0.10 & 0.18 & 0.07 \\ \hline
1000 & 0.328 & 0.157 & 0.508 & 0.335 & 0.052 & 0.166 & 0.110 \\ \hline
10000 & 0.3215 & 0.1635 & 0.5011 & 0.3544 & 0.0505 & 0.1605 & 0.1105 \\ \hline
$\vdots$ & $\vdots$ & $\vdots$ & $\vdots$ & $\vdots$ & $\vdots$ & $\vdots$ & $\vdots$ \\ \hline
$\infty$ & 0.31831\dots & 0.16666\dots & 0.50000\dots & 0.33333\dots & 0.05305\dots & 0.15915\dots & 0.10610\dots \\ \hline
\end{tabular}
\vspace{0.3cm}
\caption{Proportion of the first $n$ primes in certain classes}
\end{table}
\end{example*}

In order to obtain a bounded gaps theorem for sets of primes of the form $\mathcal{P}_{C, \alpha, \beta}$, we first need to prove an analogue of the Bombieri-Vinogradov theorem.  We adapt work of Kane~\cite{K13} to prove such a Bombieri-Vinogradov statement for the primes $\cP_{C,\alpha,\beta}$.  

Fix a finite Galois extension $L/\QQ$, a conjugacy class of its Galois group $C\subset G$, and $(a,q) =1$. Then define 
\[
\pi_C(X;q,a) := \#\Big\{p \le X:p \nmid \Delta_L, \Big[\frac{L/\QQ}{p} \Big] = C, p \equiv a \mod q \Big\}.
\] If we further specify irrationals $\alpha, \beta$, let 
\begin{equation}
\pi_{C,\alpha, \beta}(X;q,a) := \#\Big\{p\le X:p \nmid \Delta_L, \Big[\frac{L/\QQ}{p} \Big] = C, p \equiv a \mod q, p \in \mc{B}_{\alpha,\beta}\Big\}.
\end{equation} For future convenience, also define $\pi_{C,\alpha,\beta}(X) :=\# \Big\{ p\leq X: p\nmid \Delta_L, \Big[ \frac{L/\QQ}{p}\Big]=C , p \in \mc{B}_{\alpha,\beta}\}$.
Then we have the following theorem.

\begin{theorem}
\label{thm: chebeattyBV}
Suppose $L/\QQ$ is a Galois extension with a conjugacy class $C$ and $\alpha,\beta$ are real numbers with $\alpha$ irrational of finite type.
Then there exists some positive level of distribution $\theta_{\cP_{C,\alpha,\beta}} > 0$ (see \eqref{eq: LevelOfDistribution} for the exact formula) such that for any $0 < \theta < \theta_{\cP_{C,\alpha,\beta}}$ we have 
\[
\sideset{}{'}\sum_{q\le x^\theta}\max_{\substack{y \le x\\(a, q) = 1}}\Big|\pi_{C, \alpha, \beta}(y; q, a) - \dfrac{|C|}{\alpha\varphi(q)|G|}\pi(y)\Big|\ll_A \dfrac{x}{(\log x)^A}
\]
for any $A>0$, where $\sideset{}{'}\sum$ is a sum over $q\le x^\theta$ such that $K\cap\QQ(\zeta_q)=\QQ$.
\end{theorem}

Building on the work of Goldston, Pintz, and Yildirim, \cite{GPY09}, Zhang \cite{Z14} proved the existence of infinitely many pairs of primes whose difference is at most $70\cdot 10^6$ .  By developing a suitable $k$-dimensional analogue of the Selberg sieve, Maynard \cite{M15} used the Bombieri--Vinogradov theorem to reduce the upper bound on the gaps to just over $600$, and additionally proved that 
\[
\liminf_{n\rightarrow\infty}(p_{n+m}-p_n) \ll m^3e^{4m}.
\] 
(Tao independently developed the underlying sieve but came to slightly different conclusions.) 

Several authors have extended Maynard's work to study gaps between primes in special subsets, including Beatty sequences \cite{BZ16} and Chebotarev sets \cite{T14}.  We are interested in proving bounded gaps for primes lying in $\cP_{C,\alpha,\beta}$. That is there is a constant $D$ such that there are infinitely many pairs of distinct primes $q_1,q_2$ in $\cP_{C,a,b}$ such that $|q_1 - q_2| \le D$. Using the method of Maynard--Tao \cite{M15} and its extension by Thorner \cite{T14} to Chebotarev primes,
the positive level of distribution proven to exist in Theorem~\ref{thm: chebeattyBV} gives us the following bounded gaps statement.
\begin{theorem}
\label{thm: Bounded gaps}
Let $\alpha$ be a positive real number with type $\tau<\infty$.  Let $q_i$ be the $i$--th prime in $\mc{P}_{C, \alpha, \beta}$, where $\mc{P}_{C, \alpha, \beta}$ was defined in \eqref{eq: Pcab}.  For any $\sigma\in C$, let $L^{\langle\sigma\rangle}$ be the fixed field of $\sigma$.  Then there exists an effective constant $D>0$ (depending only on $\tau$ and $[L^{\langle\sigma\rangle}: \QQ]$) such that

\[
\liminf_{n\rightarrow\infty}(q_{n+m}-q_n) \le mD\Delta_L\exp(mD\Delta_L).
\]
\end{theorem}


As another application, we can combine this Maynard--Tao framework with Green--Tao theorem \cite{GT08}.  Note that the original Green--Tao theorem automatically holds for all positive density subsets of the primes.  Using work of Pintz \cite{P10} and its extension to Chebotarev sets by \cite{VW17}, our next result shows that we can find arbitrarily long arithmetic progressions within the primes in $\mc{P}_{C,\alpha,\beta}$ that are within a given distance from other primes also in $\mc{P}_{C, \alpha, \beta}$.  
\begin{corollary}\label{corollary: GreenTaoMain}
Let $L/\QQ$ be a finite Galois extension and $\mc{B}_{\alpha,\beta}$ be a Beatty sequence with $\alpha$ irrational with finite type.
Let $m$ be a positive integer and $\mathcal{H} = \{h_1,\dots, h_k\}$ be an admissible set, and suppose that 
\[
k \gg \exp\Big(\frac{2\alpha |G|\Delta_Lm}{|C|\theta \varphi(\Delta_L)}\Big),
\]
where there is an implied absolute constant. Then there is an $(m+1)$-element subset $\mathcal{H}' \subseteq \mathcal{H}$ such there exists arbitrarily long arithmetic progressions of primes in the set
\[
\{n: n + h' \in \cP_{C,\alpha,\beta} \text{ for all }h' \in \mc{H}'\}.
\]
\end{corollary}

Section \ref{pnt_proofs} contains a proof of \cref{thm:chebeatty}, which computes the natural density of $\mathcal{P}_{C,\alpha,\beta}$.  Section \ref{BV Section} contains a proof of \cref{thm: chebeattyBV}, ensuring that the primes in $\mathcal{P}_{C,\alpha,\beta}$ have a positive level of distribution.  Our proof makes use of bounds for certain exponential sums, which we prove following Kane \cite{K13} in Section 5.  In Section 6, we use \cref{thm: chebeattyBV} to prove \cref{thm: Bounded gaps}, establishing an infinitude of bounded gaps between primes in $\mathcal{P}_{C,\alpha,\beta}$.  In Section 7, we prove \cref{corollary: GreenTaoMain}, giving a common generalization of Theorem 1.3 and the work of Green and Tao.

\subsection*{Acknowledgements}
This research was conducted at the Emory University Mathematics Research Experience for Undergraduates, supported by the NSA (grant H98230-19-1-0013) and the NSF (grants 1849959, 1557960). We thank Harvard University, Stanford University, and the Asa Griggs Candler Fund for their support. The authors thank Professor Ken Ono and Jesse Thorner for their guidance and suggestions.

\section{Preliminaries}\label{prelims}

The proof of Theorem~\ref{thm:chebeatty} reduces to proving a version of the prime number theorem for primes lying in both a fixed Chebotarev class and a fixed Beatty sequence. We will prove a more precise statement by constraining these primes to an arithmetic progression $\{a+nq\}$.  Then, we will derive the explicit $q$-dependence in the error term. In Section 4 this $q$-effective error term will be crucial to the proof of our Bombieri--Vinogradov type theorem. 

First we introduce some notation. Given $\sigma \in \on{Gal}(L/\QQ)$, denote by $L^{\langle \sigma \rangle} = \{a \in L: \sigma a = a\}$. Moreover, given a set $X$ and a subset $S \subseteq X$, we denote the indicator function of $S$ by $\one_{S}$.

\begin{theorem}\label{thm: q-uniform PNT BEattyIntCheb}
Let $L/\QQ$ be a finite Galois extension and let $C \subset \on{Gal}(L/\QQ)$ be a conjugacy class.  For any $\sigma\in C$, let $d\coloneqq[L^{\langle\sigma\rangle}: \QQ]$.  Let $\kappa = 10^{-d}/(125d\tau)$ If $(a,q) = 1$, then we have
\[
\sum_{\substack{ p\leq X  \\ p\equiv a \mod q}} \one_{\mathcal{P}_{C,\alpha,\beta}}(p) \log(p) = \frac{1}{\alpha}\sum_{\substack{p\leq X \\ p\equiv a \mod q}} \mathbbm{1}_{\mathcal{P}_C}(p)\log(p) + O_{L,\alpha, \beta}(q^{d+1}X^{1 - \kappa}).
\]
\end{theorem}
\begin{remark}
Note that the error term explicitly describes the $q$-dependence. This will be vital in our application to our Bombieri--Vinogradov type result, \cref{thm: chebeattyBV}.
\end{remark}
The proof of \cref{thm: q-uniform PNT BEattyIntCheb}
relies on a condition placed on the exponential sums \[
\sum_{\substack{p \le M \\ p\equiv a \mod q}}\one_{\cP_C}(p)\log(p)e(\theta p),
\]
where $C$ is a fixed conjugacy class of the Galois group.  Using work of Kane \cite{K13}, we will show that the magnitude of this sum is $O(x^{1-\eps})$, where $\eps>0$. This will allow us to use techniques from Banks and Shparlinski \cite{BS09} to prove Theorem~\ref{thm:chebeatty} in the following section.



\subsection{Exponential sums and the work of Kane} 
Kane \cite{K13} bounds the exponential sum that we require in his proof of a variant of Vinogradov's three primes theorem for Chebotarev classes. 

Suppose that $L/\QQ$ is a finite Galois extension with $G = \Gal(L/\QQ)$, $C$ a conjugacy class of $G$, and $X$ a positive real number. Moreover, take $a < q$ with $(a,q) = (\Delta_L,q) = 1$. Adapting Kane, we define the generating function

\begin{equation}
G_{L, C, X,q,a}(\alpha) = \sum_{\substack{p\le X \\ p \equiv a \mod q}}\one_{\cP_C}(p)\log(p)e(\alpha p).
\end{equation}

To explain our bounds, we first require the following definition:

\begin{definition}
We say that a real number $\alpha$ has a \textbf{good approximation with denominator $s$} if there exists a rational number of the form $\frac{r}{s}$ with $(r, s) = 1$ satisfying $|\alpha - \frac{r}{s}| <\frac{1}{s^2}$.
\end{definition}

Our general bound, whose proof we defer to Section \ref{Kaneproofsection}, is as follows:

\begin{theorem}
\label{thm: Gexponentialsumbound}
Take $L/\QQ$ to be a finite Galois extension, relatively prime integers $a, q$, and $C \subseteq \on{Gal}(L/\QQ)$ a conjugacy class.  Pick some element $\sigma\in C$ and letf $d:= [L^{\langle \sigma\rangle}:\QQ]$. Let $\alpha$ be an irrational number with rational approximation of denominator $XB^{-1} > s > B$ for some $B > 0$. Then we have
\begin{equation}
G_{L,C,X,q,a}(\alpha) = O\Big(q^{d+1}X\cdot E(X,B)\Big),
\end{equation}
where 
\begin{equation}\label{eq:errortermexpsum}
E(X,B) = \log^2(X)B^{-10^{-d}/12} + \log^2(X)X^{-10^{-d}/60}+\log^2(X)X^{-10^{-d}/10}+\log^{2+d^2/2}(X)B^{-1/12}.
\end{equation}
In particular, for $0 < e_1 < e_2 < 1$ and a denominator in the range $s \in [X^{e_1}, X^{e_2}]$ we have $G_{L,C,X,q,a}(\alpha) = O_{e_1,e_2}(q^{d+1}X^{1 - \eta})$ for some $\eta > 0$. 
\end{theorem}

\subsection{Exponential sums for finite type $\alpha$}

The bounds we have introduced thus far depend on the denominator of a certain rational approximation. To get a uniform statement that does not depend on an approximation, we need to assume that $\alpha$ has finite type. The next bound is similar to that of \cite[Theorem 4.2]{BS09}.

\begin{lemma}
\label{lemma: exp-sum}
Let $\gamma$ be a fixed irrational number of finite type $\tau > 0$ and let $C$ be a conjugacy class of the Galois group of a finite Galois extension $L/\QQ$. Then for all real numbers $0 < \eps < \frac{1}{4 \tau}$, there exists a constant $\eta$ such that
\[
\Big|\sum_{\substack{p \le X\\
p \equiv a \mod q}}\one_{\cP_C}(p)\log(p)e(\alpha k p)\Big| = O_{L,\alpha,\beta}\Big(q^{d+1}X^{1 - \eta}\Big),
\]
for all integers $1 \le k \le X^{\eps}$ provided that $X$ is sufficiently large. Moreover, we can explicitly take $\eta = 10^{-d}/\max(125\tau,20d)$, where $d=[L:L^{\langle c\rangle}]$ for any $c \in C$.
\end{lemma}
\begin{proof}

The condition that $\alpha$ has finite type $\tau$ means that for any fixed $\rho > \tau$ there is a constant $c > 0$ such that for $s\ge 1$, we have $\|\alpha s\| > cs^{-\rho}\quad$.
Fix such a constant $\rho < \frac{1}{4\eps}$.
Now take $r/s$ to be the convergent in the continued fraction expansion of $\alpha k$ which has largest denominator $d$ not exceeding $X^{1-\eps}$, so that 
\[
\Big|\alpha k - \frac{r}{s}\Big| \le \frac{1}{s\cdot X^{1 - \eps}}.
\]
We isolate $\frac{1}{X^{1 - \eps}}$ to get
$\frac{1}{X^{1-\eps}} \ge |\alpha k s - r| \ge \|\alpha k s\| > c(ks)^{-\rho}.$ Since we are taking $k \le X^\eps$ we can rearrange to obtain
\begin{align*}
    s &\ge c^{1/\rho} k^{-1}(X^{1-\eps})^{1/\rho} \ge c^{1/\rho}X^{1/\rho}X^{\eps(-1 - 1/\rho)} \ge X^{1/(4\rho)} > X^{\eps}
\end{align*}
for sufficiently large $X$. We plug in $\alpha  = \gamma k$ in \cref{thm: Gexponentialsumbound}, which yields for $s \in [X^{\eps},X^{1 - \eps}]$ the bound
\[
\sum_{\substack{p\le X\\p \equiv a \mod q}}\one_{\cP_{C}}(p)\log(p)e(\alpha p) =  O(q^{d+1}X\cdot E(X,X^{\eps})) = O(q^{d+1}X^{1-\eta}).
\]
We now explicitly compute $\eta$.  Specifically, in \cref{thm: Gexponentialsumbound}, take $\eps = \frac{1}{5\tau}$.  Then in the region $q \in [X^{1/(4 \tau)}, X^{1 - \eps}] \subseteq [X^{\eps}, X^{1 - \eps}]$, we have that by setting $B = X^{\eps/2}$ that  $G_{L,C, X,q,a}(\alpha)$ is 
\begin{align*}
O\Big(q^{d+1}X\Big(\log^2(X)X^{-10^{-d}/(120\tau)} + \log^2(X)X^{-10^{-d}/60}+\log^2(X)X^{-10^{-d}/(10d)}+\log^{2+d^2/2}(X)X^{-1/120\tau}\Big)\Big)
\\= O\Big(q^{d+1}X\Big(\log^2(X)\Big(X^{-10^{-d}/(120\tau)} + X^{-10^{-d}/(10d)}\Big)\Big)\Big)
= O\Big(q^{d+1}X^{1-10^{-d}/b + \delta}\Big)
\end{align*}
for any small $\delta > 0$ and $b= \max(10d, 120\tau)$.
Therefore, taking $\eta = 10^{-d}/\max(20d, 125\tau)$ yields an explicit bound.
\end{proof}






\subsection{Facts about irrational sequences}
We begin with some preliminary definitions regarding irrational numbers. This presentation follows section 3 of \cite{BS09}.  The first tool we will need is a bound on discrepancies.

Define 
\begin{equation}
D_{\gamma,\delta}(M)\coloneqq \sup_{\mc{I}\subseteq [0, 1)}\Big|\dfrac{V(\cI, M)}{M} - |\cI|\Big|,
\end{equation}
where $V(\cI, M)$ is the number of positive integers $m\le M$ such that $\{\gamma m + \delta\}\in\cI$ and the supremum is taken over all subintervals $\cI = (r_1, r_2)$ of the interval $[0, 1)$.

\begin{lemma}[\cite{KN06}, Chapter 2, Theorem 3.2]
\label{lemma:discrepancy}

Let $\gamma$ be a fixed irrational number of finite type $\tau<\infty$.  Then for all $\delta\in\RR$ we have that
\[
D_{\gamma, \delta}(M)\le M^{-\frac{1}{\tau}+o(1)} \qquad (M\rightarrow\infty),
\]
where the function implied by $o(\cdot)$ depends only on $\gamma$.
\end{lemma}

The key to analyzing Beatty sequences arises from the fact that they have cleanly stated indicator functions. This next elementary lemma allows us to rewrite the indicator function of a Beatty sequence in a tractable form.
\begin{lemma}[{\cite[Lemma 3.2]{BS09}}]
\label{lemma:elementary}
Let $\alpha, \beta$ be in $\RR$ with $\alpha > 1$.  Then an integer $m$ has the form $m=\lfloor\alpha n+\beta\rfloor$ for some integer $n$ if and only if
\[
0<\{\alpha^{-1}(m-\beta + 1)\}\le\alpha^{-1}.
\]
The value of $n$ is determined uniquely by $m$.
\end{lemma}

\subsection{Intersecting general sets of primes with Beatty sequences}

One might ask if we can replace the Chebotarev condition with other sets of primes $S$.
Indeed, if the corresponding exponential sums are bounded by $O(X^{1 - \kappa})$ for some $\kappa > 0$, we could prove an analogous prime number theorem for primes that are both in $S$ and a fixed Beatty sequence.

\begin{remark}
One also might ask whether the power-saving exponential sum bound can be replaced with a weaker bound, such as $o\Big(\frac{X}{(\log X)^A}\Big)$ for any $A$. It is likely that we could prove an analogous prime number theorem. However, these cases do not apply to Chebotarev sets, and making a general statement about such primes may be difficult. However, we will see in that a power-saving bound is essential for proving a Bombieri--Vinogradov type theorem.
\end{remark}

\begin{remark}
Banks and Shparlinski also proved a prime number theorem for primes in the sequences $\mathcal{C}_{n,\alpha, \beta, q, a} = q\lfloor \alpha n + \beta \rfloor + a, (a,q) = 1
.$ Their results show that the natural density of this set is $ \delta( \mathcal{C}_{\alpha, \beta, q, a})= \frac{1}{\varphi(q)\alpha}.$ One can easily modify our proofs in order to obtain prime number theorem and density statements for primes in $T_{\alpha,\beta,q, a}:= \cP_{C} \cap \mathcal{C}_{\alpha, \beta,q,a}$. The proofs of Bombieri--Vinogradov and bounded gaps will also be the same, since the error term will still be power-saving.
\end{remark}

\section{Proofs of Theorem~\ref{thm:chebeatty} and Theorem~\ref{thm: q-uniform PNT BEattyIntCheb}}\label{pnt_proofs}
We now prove \cref{thm: q-uniform PNT BEattyIntCheb}, which implies \cref{thm:chebeatty}. Before proving this statement, we need to detail Vinogradov's theorem on the indicator function of a Beatty sequence.
\begin{theorem}[{\cite[Chapter~I, Lemma~12]{V04}}]
\label{thm: BeattyVinogradov}
Let $\gamma\in (0, 1)$ be a real number and let $\psi(x)$ be the periodic function with period one for which 
\[
\psi(x) = \begin{cases} 
1 &\text{if } 0<x\le\gamma; \\ 
0 &\text{if } \gamma<x\le 1. 
\end{cases}
\]
For any $\Delta$ such that 
\[
0 < \Delta < \frac{1}{8}\quad \text{ and }\quad \Delta \le \frac{1}{2}\min(\gamma, 1 - \gamma),
\]
there is a real-valued function with the following properties
\begin{enumerate}
    \item $\psi_\Delta(x)$ is periodic of period one;
    \item $0 \le \psi_\Delta(x) \le 1$ for all $x \in \RR$;
    \item $\psi_\Delta(x) = \psi(x)$ if $\Delta \le x \le \gamma - \Delta$ or $\gamma + \Delta \le x \le 1 - \Delta$.
    \item 
    $\psi_\Delta(x)$ has a Fourier series
    \[
    \psi_\Delta(x) = \gamma + \sum_{k = 1}^{\infty}(g_ke(kx) + h_ke(-kx))
    \]
    with uniform bounds on the coefficients 
    \[
    \max(|g_k|, |h_k|) \le \min\Big(\frac{2}{\pi k}, \frac{2}{\pi^2k^{2}\Delta}\Big)\text{ for all } k \ge 1.\]
\end{enumerate}
\end{theorem}
Combining this Fourier expansion with our bounds on exponential sums allows us to prove a version of the prime number theorem for primes in $\mc{P}_{C, \alpha, \beta}$.

\begin{proof}[Proof of \cref{thm: q-uniform PNT BEattyIntCheb}]

Define 
\[S_{C,\alpha,\beta, q,a}(X) \coloneqq \sum_{\substack{p \le X\\ p\equiv a \mod q}}\one_{\cP_{C}}(p)\one_{\cP_{\alpha,\beta}}(p)\log p.
\] 
Let $\gamma = \alpha^{-1}$ and $\delta = \alpha^{-1}(1 - \beta)$.
By Lemma~\ref{lemma:elementary}, $p = \lfloor \alpha n + \beta \rfloor$ if and only if $0 < \{\gamma p + \delta\} \le \gamma $, which implies that
\begin{align*}
S_{C,\alpha, \beta, q, a}(X) = \sum_{\substack{p \le X\\ p\equiv a \mod q}}\one_{\cP_{C}}(p)\one_{\cP_{\alpha,\beta}}(p)\log p
 &= \sum_{\substack{p \le X\\ p\equiv a \mod q}}\one_{\cP_{C}}(p)\one_{(0,\gamma]}(\gamma p + \delta)\log p + O(1) \\
&= \sum_{\substack{p \le X\\ p\equiv a \mod q}}\one_{\cP_{C}}(p)\psi(\gamma p + \delta)\log p + O(1).
\end{align*}

Then we can replace the the indicator function $\psi$ with the Fourier expansion described in Theorem~\ref{thm: BeattyVinogradov} at a cost of
\[
S_{C,\alpha,\beta,q,a}(X) = \sum_{\substack{p\le X \\ p\equiv a \mod q}}\one_{\cP_C}(p)\log(p)\psi_\Delta(\gamma p + \delta) + O(1 + \mathcal{V}(\mathcal{I},X)\log X),
\]
where we are taking $\mathcal{V}(\mathcal{I},X)$ to be the set of integers $m \le X$ with the prescribed fractional part lying inside
\[
\mathcal{I} = [0,\Delta) \cup (\gamma - \Delta, \gamma + \Delta) \cup (1 - \Delta, 1).
\]

We have $|\mathcal{I}| \le 4\Delta$, so by the definition of discrepancy and Lemma~\ref{lemma:discrepancy} we have that 
 \[
 \mathcal{V}(\mathcal{I},X) = O(\Delta X + X^{1/2-1/\tau}).
 \]

Now to get a handle on $S_{C,\alpha, \beta, q,a}$, we use the Fourier expansion and get 
\begin{align*}
    \sum_{\substack{p\le X\\ p\equiv a \mod q}}\one_{\cP_{C}}(p)\log(p)\psi_\Delta(\gamma p + \delta) = \\\gamma\sum_{\substack{p\le X \\ p\equiv a \mod q}}\one_{\cP_{C}}(p)\log(p) &+ \sum_{k = 1}^\infty g_ke(\delta k)\sum_{\substack{p\le X\\ p\equiv a \mod q}}\one_{\cP_{C}}(p)\log(p)e(\gamma k p)\\ &+ \sum_{k = 1}^\infty h_ke(-\delta k)\sum_{\substack{p\le X\\p\equiv a \mod q}}\one_{\cP_{C}}(p)\log(p)e(-\gamma k p).
 \end{align*}

Now we are in a position to use the bounds on exponential sums we derived earlier. By \cref{thm: Gexponentialsumbound}, for 
$1 \le k \le X^\eps$ where $\eps < \frac{1}{4\tau}$ we have 
\[
\sum_{k \le X^\eps}g_ke(\delta k)\sum_{\substack{p \le X\\p\equiv a \mod q}}\one_{\cP_{C}}(p)\log(p)e(\gamma k p) =O\Big(\dfrac{q^{d+1}X}{E(X,X^{\eps})} \sum_{k \le X^{\eps}}k^{-1}\Big) = O\Big(\dfrac{q^{d+1}X\log(X^{\eps})}{E(X,X^{\eps})}\Big),
\]
recalling the definition of $E(X,B)$ from \eqref{eq:errortermexpsum}. By~\cref{lemma: exp-sum} there exists some $\eta > 0$ for which we may take $E(M, M^\eps)=M^\eta$.  Similarly, we have

\[
\sum_{k \le M^\eps}h_ke(\delta k)\sum_{\substack{p\le X\\p\equiv a \mod q}}\one_{\cP_{C}}(p)\log(p)e(\gamma k p) = O\Big(q^{d+1}\dfrac{X\log(X^\eps))}{E(X,X^{\eps})}\Big).
\]
For the larger terms, we can use the trivial bound 
\[
\Big|\sum_{\substack{p\le X\\p\equiv a \mod q}}\one_{\cP_{C}}(p)\log(p)e(\gamma k m)\Big| \le \sum_{\substack{p\le X\\p\equiv a \mod q}} \log(p) = O(X),
\] along with the bounds on the Fourier coefficients to get that
\[
\sum_{k > X^{\eps}}g_ke(\delta k)\sum_{\substack{ p\le X\\p\equiv a \mod q}}\one_{\cP_C}(p)\log(p)e(\gamma k p) =O\Big( X \sum_{k > X^{\eps}}k^{-2}\Delta^{-1}\Big) =O\Big( \dfrac{X}{X^{\eps}\Delta}\Big).
\]
This yields
\begin{align*}
    \sum_{\substack{p \le X\\p\equiv a \mod q}}\Lambda(m)\psi_{\Delta}(\gamma p + \delta) = \gamma \sum_{\substack{p \le X\\p\equiv a \mod q}}\one_{\cP_C}\log(p) +O\Big(\dfrac{q^{d+1}X\log(X^{\eps})}{E(X,X^{\eps})} +\dfrac{X}{E(X,X^{\eps})\Delta}\Big).
\end{align*}
Thus we have
\[
S_{C,\alpha, \beta, q, a}(X) = \gamma\sum_{\substack{p\le X\\p\equiv a \mod q}} \one_{\cP_C}\log(p)+O\Big(\Delta X\log X + X^{1/2-\eps}\log X +\dfrac{q^{d+1}X\log(X^{\eps})}{E(X,X^{\eps})} +\dfrac{X^{1-\eps}}{\Delta}\Big),
\]
where the constant implied by $O(\cdot)$ depends only on $\alpha$ and $L$.

Finally, we plug in $E(X,X^{\eps}) \ll X^\eta$ and set $\Delta = \min\{X^{-2\eta}, \frac{10}{81}, \frac{\gamma}{2}, \frac{1-\gamma}{2}\}$, where we note that in our case $10\eta < \eps$.  This gives 
\begin{align*}
&S_{C,\alpha, \beta, q, a}(X)\\
&= \gamma\sum_{\substack{p \le X\\p\equiv a \mod q}} \one_{\cP_C}(p)\log(p)+O\Big(q^{d+1}\Big(X^{1+2\eta - \eps}\log X + X^{1/2-\eps}\log X +X^{1-\eta}\log X +X^{1-2\eta}\Big)\Big)\\
&= \frac{1}{\alpha}\sum_{\substack{p\le X\\p\equiv a \mod q}} \one_{\cP}(p)\log(p) +  O\Big(q^{d+1}
X^{1 - \eta}\Big)
\end{align*}
where again the constant implied by $O(\cdot)$ depends only on $\alpha$ and $L$, which implies the desired result.
\end{proof} 
\begin{remark}
Note that if $q^{d+1}X^{1 - \kappa} \gg X$ then this bound is trivial. Thus, in the theorem statement we require $q < X^{\kappa/(d + 1) - \eps}$ for any $\eps > 0$ to ensure the error term is still power saving.
\end{remark}
Recalling that the exponential sums yield an upper bound of $M^{1- \eta}$ for $\eta = 10^{-d}/\max(125\tau, 20 d)$, we note that for this error term we can choose the same power-saving $\eta$, which for convenience of writing we can take to be $\kappa = 10^{-d}/(125d\tau)$.

Finally, applying a standard partial summation argument to Theorem~\ref{thm: q-uniform PNT BEattyIntCheb} gives the following $q$-effective version of the prime number theorem.

\begin{theorem}
\label{thm: q-effective PNTBeattyIntCheb}
For $(q,\Delta_L) = 1$ we have the asymptotic
\[
\pi_{C,\alpha,\beta}(X; q, a) = \frac{1}{\alpha}\pi_{C}(X;q,a) + O_L(q^{d+1}X^{1 - \kappa}),
\]
where $\kappa = 10^{-d}/(125d\tau)$, $q < X^{\kappa/(d + 1)-\eps}$ for any $\eps > 0$, and $X > X_0(\alpha,\beta, C, L)$ for a constant $X_0$.
\end{theorem}

\section{A Bombieri--Vinogradov type theorem}\label{BV Section}

In this section we prove our analogue of the Bombieri--Vinogradov theorem.  There are two ingredients for this proof: a Bombieri--Vinogradov type theorem for Chebotarev classes and the $q$-effective version of the prime number theorem proven in the preceding section.  The first of these was done by Murty and Murty \cite{MM87}.  We explain their result below.  

We note that Chebotarev classes in the cyclcotomic field $\QQ(\zeta_q)$ are equivalent to residue classes $a\mod q$ where $(a,q)=1$.
This point is precisely why we need a $q$-effective version of the prime number theorem; we need to explicitly extract the error term for the field $L(\zeta_q)$ in terms of the error term for $L$.

For Chebotarev classes,
Murty and Murty prove the following Bombieri--Vinogradov type theorem.
\begin{theorem}
\cite{MM87}
\label{thm: MurtyMurty}
For $Q=x^{\theta-\eps}$ with some $\theta \ge \min\{\frac{2}{|G|}, \frac{1}{2}\}$ and any $\eps>0$, we have 
\[
\sideset{}{'}\sum_{q\le Q}\max_{\substack{y \le x\\(a, q) = 1}}\Big|\pi_{C}(y, q, a) - \dfrac{|C|}{\varphi(q)|G|}\pi(y)\Big|=O_A\Big( \dfrac{x}{(\log x)^A}\Big),
\]
where the sum is over $q$ such that $K\cap\QQ(\zeta_q)=\QQ$.
\end{theorem}

Combining this statement with our $q$-effective prime number theorem, \cref{thm: q-effective PNTBeattyIntCheb} allows us to conclude our proof of \cref{thm: chebeattyBV}. For $\tau$ the type of $\alpha$ and $d$ the dimension of the fixed of an element $\sigma \in C$, we will end up taking the following level of distribution:
\begin{equation}
\label{eq: LevelOfDistribution}
\theta_{\cP_{C,\alpha,\beta}} :=  \frac{1}{125\cdot 10^{d}\cdot d(d+1)\tau}.
\end{equation}

\begin{proof}[Proof of \cref{thm: chebeattyBV}]
We have by Theorem~\ref{thm: q-effective PNTBeattyIntCheb} that $\pi_{C, \alpha, \beta}(y, q, a) =  \frac{1}{\alpha}\pi_C(y,q,a) + O_L(q^{d + 2}y^{1 - \kappa})$ where we are taking $\kappa = 10^{-d}/(125 d \tau)$ where $d$ is the dimension of the fixed field of an element $\sigma \in C$. First set $\theta_{\cP_{C,\alpha,\beta}} = \min(\kappa/(d + 1) , \theta_C) = \kappa/(d+1)$ where $\theta_C$ is the Chebotarev $\theta$ as in Theorem~\ref{thm: MurtyMurty}, and $\kappa = \kappa_{K(\zeta_{Q})}$ as in \cref{thm: q-effective PNTBeattyIntCheb}. Then set $\theta = \theta_{\cP_{C,\alpha,\beta}} - \eps$ for any $\eps > 0$.  We take maxima and sum over moduli $q \le Q := x^{\theta}$ with $K \cap \QQ[\zeta_q] = \QQ$. Then for some constant $A_L$ depending only on $L$, we have by the triangle inequality that
\begin{align*}
    \sum_{q\le x^\theta}&\max_{\substack{y \le x\\(a, q) = 1}}\Big|\pi_{C, \alpha, \beta}(y; q, a) - \dfrac{|C|}{\varphi(q)|G|\alpha}\pi(y)\Big| \\
    &= \sum_{q\le x^\theta}\max_{\substack{y \le x\\(a, q) = 1}}\Big|\frac{1}{\alpha}\pi_C(y; q, a) - \dfrac{|C|}{\alpha\varphi(q)|G|}\pi(y) + O_L(q^{d+1}y^{1-\kappa})\Big| \\
    &=\frac{1}{\alpha}\sum_{q\le x^\theta}\max_{\substack{y \le x\\(a, q) = 1}}\Big|\pi_C(y; q, a) - \dfrac{|C|}{\varphi(q)|G|}\pi(y)\Big| +\sum_{q\le x^\theta}\Big|A_Lq^{d+1}x^{1 - \kappa}\Big|
    \\
    &=O_A\Big( \dfrac{x}{\alpha(\log x)^A} + x^{\kappa - (d+1)\eps}O_L(x^{1-\kappa})\Big)\\
    &=O\Big( \dfrac{x}{(\log x)^A}\Big), \numberthis
\end{align*}
where we use \cref{thm: MurtyMurty} and \cref{thm: q-effective PNTBeattyIntCheb} to bound each of the two terms term.
\end{proof}

\section{Bounding Exponential Sums with Chebotarev Conditions}\label{Kaneproofsection}

We now turn our attention to proving \cref{thm: Gexponentialsumbound}, which bounds the magnitude of the sum $G_{L,C,X,q,a}(\alpha)$. By the discussion in the foregoing section, this problem reduces to placing a $q$-effective bound on the sum defined below.

\begin{definition}
\label{definition: twisted G}
Suppose that $L/\QQ$ is a finite Galois extension with $G=\on{Gal}(L/\QQ)$, $C$ a conjugacy class of $G$, $\chi$ a Dirichlet character with modulus $q$, and $X$ a positive real number.  We define the function 
\[
G_{L, C, X, \chi}(\alpha) = \sum_{\substack{p\le X}}\one_{\cP_C}(p)\log(p)e(\alpha p)\chi(p).
\]
\end{definition}

Then by orthogonality of Dirichlet characters, we have that
\begin{equation}
\label{eq: DirichletOrthogonality}
G_{L, C, X,q,a}(\alpha) = \frac{1}{\varphi(q)}\sum_{\chi}\overline{\chi(a)}G_{L,C,X,\chi}(\alpha),
\end{equation}
so it suffices to bound $G_{L,C,X,\chi}$ for all $\chi$ of modulus $q$.  We can further use orthogonality of Hecke characters (see \cite[Chapter~VII]{N99} for more on Hecke characters) to decompose these sums into certain functions $F_{K, \xi, X,\chi}$. First, given a number field $K/\QQ$, we need the \emph{von Mangoldt function} on integral ideals of $\mathcal{O}_K$ given by
\[
\Lambda_K(\mathfrak{a}):= 
\begin{cases}
\log \on{N}\mathfrak{p}& \mathfrak{a} = \mathfrak{p}^r, r \ge 1\\
0 & \text{ otherwise}.
\end{cases}
\]
\begin{definition}
\label{definition: twisted F}
If $K/\QQ$ is a number field, $\xi$ a Gr\"{o}ssencharacter of $K$, $\chi$ a Dirichlet character with modulus $q$, and $X$ a positive number, define the function 
\[
F_{K, \xi, X, \chi}(\alpha) = \sum_{N(\mathfrak{a})\le X}\Lambda_K(\mathfrak{a})\xi(\mathfrak{a})e(\alpha \on{N}(\mathfrak{a}))\chi(\on{N}(\mathfrak{a})),
\]
where the sum above is over ideals $\mathfrak{a}$ of $L$ with norm at most $X$.
\end{definition}
The purpose of switching from $\one_{\cP_{C}}(\mathfrak{p})\log \on{N}\mathfrak{p}$ to $\Lambda_K(\mathfrak{a})$ is for a future application of Vaughan's identity in the proof of \cref{proposition: qdependentKanesum}. The next proposition, which is a variant of  \cite[Proposition 6]{K13}, gives us an orthogonality statement for the $F_{K,C,X,\chi}$.
\begin{proposition}\label{prop: our Kane 6}

Let $L/\QQ$ be a finite Galois extension and $C$ a conjugacy class of $G=\Gal(L/\QQ)$. Choose a representative $\sigma \in C$ and let $K:= L^{\langle\sigma \rangle}$.  Then we have that 
\[
G_{L, C, X,\chi} = \dfrac{|C|}{|G|}\Big(\sum_{\xi}\overline{\xi(c)}F_{K,\xi, X, \chi}(\alpha)\Big) + O\Big(\sqrt{X}\log X\Big),
\]
where the sum is over characters $\xi$ of the subgroup $\langle c\rangle \subset G$, which can be thought of as (ray class) characters of $L$.
\end{proposition}

\begin{proof}
The proof is identical to that of \cite[Proposition 6]{K13}, with the key idea being that ideals with norm equal to a larger prime power contribute $O\Big(\sqrt{X}\log X\Big)$ to the sum. The only difference is that both sides of the equality are multiplied by a Dirichlet character. 
\end{proof}

In his work, Kane proves a proposition concerning the value of $F_{K,\xi,X} := F_{K, \xi, X,\chi_0}$ where $\chi_0$ is the trivial character of modulus $1$ (the identity). Our $q$-effective version of this proposition is very similar, except with an extra factor of $q^{d+1}$.


\begin{proposition}\label{proposition: qdependentKanesum}
Let $K$ be a degree $d$ number field, $\xi$ a Gr\"{o}ssencharacter, and $\chi$ a Dirichlet character with modulus $q$. Let $\alpha$ be an irrational number with rational approximation of denominator $s$ with $XB^{-1} > s > B$ for some $B > 0$. Then we have the estimate $F_{K,\xi,X,\chi}(\alpha)=$
\[
O\Big(q^{d+1}X\Big(\log^2(X)B^{-10^{-d}/12} + \log^2(X)X^{-10^{-d}/60}+\log^2(X)X^{-10^{-d}/10}+\log^{2+d^2/2}(X)B^{-1/12}\Big)\Big),
\]
where the asymptotic constant may depend on $K$ and $\xi$, but not on $X$, $s$, $q$, $B$, or $\alpha$.
\end{proposition}
In order to arrive at this result, we will need $q$-analogues of many results of Kane. We follow his argument, making the necessary adjustments where needed.

\subsection{$q$-effective versions of results of Kane}
With Proposition~\ref{proposition: qdependentKanesum} as a goal, we cite and adapt several lemmas from Kane \cite{K13}.  The first needs no modification.

\begin{lemma}
\label{lemma: Kane 18}
\emph{\cite[Lemma 18]{K13}}
Let $X, A, C$ be positive integers.  Let $\alpha$ be a real number with rational approximation of denominator $s$.  Suppose that for some $B>2A$, that $XB^{-1}>s>B$.  Then there exists a set $S$ of natural numbers so that
\begin{itemize}
    \item Elements of $S$ are of size at least $\Omega(BA^{-1})$.
    \item The sum of the reciprocals of the elements of $S$ is $O(A^2B^{-1}+X^{-1}A^4C)$.
    \item For all positive integers $n\le C$, either $n$ is a multiple of some element of $S$ or $n\alpha$ has a rational approximation with some denominator $s'$ with $XA^{-1}n^{-1}>s'>A$.
\end{itemize}
\end{lemma}

The following two lemmas are versions of Lemmas 19 and 20 in \cite{K13}, twisted by a Dirichlet character.

\begin{lemma}
\label{lemma: Our 19}
Fix $K$ a number field.  Let $n$ be a positive integer, let $\chi$ be a Dirichlet character of modulus $q$, and let $X$ and $\eps$ be positive real numbers.  Then for any $\eps>0$, we have that
\begin{align*}
    \sum_{\substack{n|\on{N}(\mathfrak{a}) \\ \on{N}(\mathfrak{a})<X}}\dfrac{\chi(\on{N}(\mathfrak{a}))}{\on{N}(\mathfrak{a})} &= O\Big(\dfrac{\log(X)n^{\eps}}{n}\Big), \\ 
    \sum_{\substack{n|\on{N}(\mathfrak{ab}) \\ \on{N}(\mathfrak{ab})<X}}\dfrac{\chi(\on{N}(\mathfrak{ab}))}{\on{N}(\mathfrak{ab})} &= O\Big(\dfrac{\log^2(X)n^{\eps}}{n}\Big),
\end{align*}
where the implied constant depends on $K$ and $\eps$, but nothing else. (The first sum above is over ideals $\mathfrak{a}$ such that $n|\on{N}(\mathfrak{a})$ and $\on{N}(\mathfrak{a})\le X$.  The second sum is over pairs of ideals $\mathfrak{a}$ and $\mathfrak{b}$, such that $\on{N}(\mathfrak{a}\cdot\mathfrak{b})$ satisfies the same conditions.)
\end{lemma}
\begin{proof}
The statement of \cite[Lemma 19]{K13} is the same as the one we wish to prove except that $\chi$ is specialized to the trivial character; that is, the numerator is $1$.  The magnitude of the left hand side is maximized when the trivial character is taken, so the result holds by Lemma 19 of \cite{K13}.
\end{proof}

\begin{lemma}
\label{lemma: Our 20}
Pick a positive integer $X$.  Let $[X]=\{1, 2, \ldots, X\}$.  Let $P$ be a rational polynomial with leading term $cx^k$ for some integer $c\neq 0$, and let $Q$ be an integral polynomial.  Let $\alpha$ be a real number with a rational approximation of denominator $s$.  Let $\chi$ be a Dirichlet character with modulus $q$.  Then 
\[
\Big|\sum_{x\in [X]}e(\alpha P(x))\chi(Q(x))\Big| =O\Bigg( q^k|c|X\Big(\dfrac{1}{s}+\dfrac{q}{X}+\dfrac{sq^k}{X^k}\Big)^{10^{-k}}\Bigg),
\]
where the implied constant depends only on $k$.
\end{lemma}
\begin{proof}
Call the sum in question $S(\alpha, \chi)$. Since $Q$ is an integral polynomial, we have that $Q(nq + a) \equiv Q(a)\pmod q$.
Thus, we can break up the sum in question in accordance with the residue class of $x$ modulo $q$:
\[
|S(\alpha, \chi)| = \Big|\sum_{x\in [X]}e(\alpha P(x))\chi(Q(x))\Big| = \Big|\sum_{a \in (\ZZ/q\ZZ)^\times}\chi(Q(a))\sum_{\substack{x \le \lfloor\frac{X - a}{q}\rfloor}}e(\alpha P(qx + a)) \Big|.
\]
Now, upon applying Lemma 20 of Kane and noting that the leading coefficient of $P(qx + a)$ is $q^kc$, we get that 
\begin{align*}
S(\alpha, \chi) &\le \sum_{a \in (\ZZ/q\ZZ)^\times}\Big|\chi(Q(a))\sum_{\substack{x \le \lfloor\frac{X - a}{q}\rfloor}}e(\alpha P(qx + a))\Big| \\
&=O_L\Bigg( \sum_{a \in (\ZZ/q\ZZ)^\times} |q^kc|\cdot \frac{X - a}{q}\Big(\frac{1}{s} + \frac{q}{X- a} + \frac{sq^k}{(X - a)^k}\Big)^{10^{-k}}\Bigg)\\
&=O_L\Bigg( q^k |c| X \Big(\frac{1}{s} + \frac{q}{X} + \frac{sq^k}{X^k}\Big)^{10^{-k}}\Bigg).
\end{align*}
\end{proof}

These lemma are vital for the main theoretical input of Kane, which is a geometry of numbers argument used to bound an exponential sum without a Von Mangoldt weighting. We now prove the $q$-analogue of this result.

\begin{proposition}
\label{proposition: Our 21}
Assume that $K$ is a number field of degree $d$, $\xi$ is a Gr\"{o}ssencharacter of modulus $\mathfrak{m}$, and $\chi$ is a Dirichlet character of modulus $q$.  Then for any positive number $X$ and a real number $\theta$ with a rational approximation of denominator $s$, we have the bound
\begin{equation}\label{eq:what we want}
    \Big|\sum_{\on{N}(\mathfrak{a})\leq X}\xi(\mathfrak{a})e(\alpha \on{N}(\mathfrak{a}))\chi(\on{N}(\mathfrak{a}))\Big| = O_{L,\xi}\Big(q^{d+1}X\Big(\frac{1}{s} + \frac{1}{X^{1/d}} + \frac{s}{X}\Big)^{10^{-d}/2}\Big).
\end{equation}
\end{proposition}
\begin{proof}
 This proof follows much in the same way as \cite[Lemma 21]{K13}. We will recall the setup of that proof, and indicate where we make modifications. Kane's proof requires an application of Lemma 20 in his paper, for which we have the $q$-effective analogue \cref{lemma: Our 20}. Our bound has the extra condition that the polynomial $q(x)$ inside the character $\chi$ must be \emph{integral}, not just rational. No other serious modifications are required. First, divide the sum into ray classes of modulus $\mathfrak{m}$. There are only finitely many such ray classes, so it suffices to prove our result for a single class. In other words, we bound the sum over $\mathfrak{a}$ in a given class. 

Fix an integral ideal representative $\mathfrak{a}_0$ of a ray class. Setting $K_{\mathfrak{m}}^1$ be the elements of $K$ that are congruent to $1\mod \mathfrak{m}$, every integral ideal in the class of $\mathfrak{a}_0$ can be written as $b\mathfrak{a}_0$ for some $b\in K_{\mathfrak{m}}^1\cap \mathfrak{a}_0^{-1}$. The representation is unique up to multiplication by an element in $\mathcal{O}_K^{\ast}\cap K_{\mathfrak{m}}^1$. By the multiplicativity of  Gr\"{o}ssencharacters, $\xi(b\mathfrak{a}_0)=\xi(b)\xi(\mathfrak{a}_0)$.  Because $\xi$ has modulus $\mathfrak{m}$ and $b\in K_{\mathfrak{m}}^1$, we can approximate $\xi(b)=\psi(b)$ by a continuous character $\psi: (K\otimes \RR)^\ast\rightarrow \mathbb{C}^\ast$, where $\psi(\mathcal{O}_K^*\cap K_{\mathfrak{m}}^1)=1$. The multiplicativity of the norm implies that $\on{N}_{K/\QQ}(b\mathfrak{a}_0)=|\on{N}_{K/\QQ}(b)|\on{N}_{K/\QQ}(\mathfrak{a}_0)$.

Kane then implements a geometry of numbers argument. First, note that $T:= K_{\mathfrak{m}}^1\cap \mathfrak{a}_0^{-1}$ is a translate of a lattice in $K\otimes \RR$. Kane uses the fact that $\on{N}_{K/\QQ}(b)$ is a degree $d$ polynomial with rational coefficients in terms of a basis for this lattice.

For our situation, however, we want an integral polynomial. In fact, upon fixing a free basis for $\mc{O}_K$, we have that $\on{N}_{K/\QQ}$ has integral coefficients on $K$ (not on the lattice). Call this polynomial $P(b)$ where $b = (x_1,\dots, x_n)$ is in the chosen integral basis for $K/\QQ$. To make better sense of the norm applied to $\mathfrak{a}_0b$, we choose $k$ such that $\mathfrak{a}_0^k$ is trivial in the ray class group, and rewrite 
$\on{N}_{K/\QQ}(\mathfrak{a}_0b) = \on{N}_{K/\QQ}(g_{k,b})/\on{N}_{K/\QQ}(\mathfrak{a}_0^{k-1})$ where we choose a generator $(\mathfrak{a}_0^kb) = (g_{k,b})$ for the principal ideal.

The sum that we need to compute is over all of the $b$ in the lattice with norm not exceeding $X/\on{N}_{K/\QQ}(\mathfrak{a}_0)$ in a fundamental domain of the action of $\mathcal{O}_K^{\ast}\cap K_{\mathfrak{m}}^1$.  First take $D$ to be a fundamental domain for the action of $\mathcal{O}_K^{\ast}\cap K_{\mathfrak{m}}^1$ in $(K \otimes \RR)^{\ast}$. Then letting $R := D \cdot (0, (X/\on{N}_{K/\QQ}(\mathfrak{a}_0))^{1/d})$, the over the desired fundamental
\begin{align} \label{eq:Kane21suminquestion}\frac{\xi(\mathfrak{a}_0)}{\chi(N_{K/\QQ}(\mathfrak{a}_0^{k-1}))} \sum_{b\in T\cap R} \psi(b) e(\alpha|P(b)|)\chi(|P(g_{k,b})|), 
\end{align}
where we note $P(b)$ is an integral polynomial on the fixed basis of $\mathcal{O}_K$. Kane notes that $P(b)$ will have a constant sign on connected components of $R$, since $P$ extends to a continuous function on $R$.  Therefore, by restricting the sum to a single connected component of $R$, we can ignore the absolute value of $P$ taken in the sum.

The main idea is to reduce \eqref{eq:Kane21suminquestion} to a sum in $1$ dimension, which can be handled by Lemma~\ref{lemma: Our 20}. In this case, we want our summands to be of the form $e(\alpha p(x))\chi(q(x))$ for $p \in \QQ[x]$ with an integer leading term and $q(x) \in \ZZ[x]$. In order to do this, pick a vector $v \in K$, $t + v \in T$ for all $t \in T$. Then $P(b+vx)$ is one variable degree $d$ polynomial in $n$ whose rational leading coefficient does not depend on $b$. By replacing $v$ with a positive integral multiple of itself, we can assume that the leading coefficient is a non-zero integer.
Moreover, first choosing an integral generator $(v_{k}) = \mathfrak{a}_0^kv$, we have that the polynomial inside the character $\chi$ is $Q(b + vx) := P(\mathfrak{a}_0^k(b + vx)) = P(g_{k,b} + v_kx)$. Since $g_{k,b},v_k$ are algebraic integers and $P$ is an integer polynomial, $Q(b + vx)$ also has integer coefficients.

With this choice, the rest of the argument directly follows Kane. First fixing an integer $Y= \Theta(X^{1/2d})$, he defines a \emph{line} in $T$ to be a subset of $T$ of the form $\{b+v, b+2v,\dots,b+Yv\}$ for $b\in T$.  Each element of $T$ is contained in $Y$ lines, so we can rewrite \eqref{eq:Kane21suminquestion} as \begin{align}\label{eq:linessum}
 \frac{\xi(\mathfrak{a}_0)}{\chi(\on{N}_{K/\QQ}(\mathfrak{a}_0^{k-1}))Y}\sum_{\textrm{lines }N} \sum_{b\in N\cap R} \psi(b) e(\alpha P (b))\chi(Q(b)). 
 \end{align}

Kane breaks up the sum based on whether $N$ is contained in $R$. 
For lines not contained in $R$, Kane's argument applies and we get a bound $O(X^{1 - 1/d}Y)$.  When a line $N$ contained in $R$, his arguments show the contribution of the inner summand in \eqref{eq:linessum} is
\[ \min(Y^2|(b+v)^{-1}|, Y) + \sum_{n=1}^Y e(\alpha p_b(n))\chi(q_b(n)).\]

By Lemma \ref{lemma: Our 20}, the second term is \[ q^k|c|Y\Big(\dfrac{1}{s}+\dfrac{q}{Y}+\dfrac{sq^k}{Y^k}\Big)^{10^{-k}}.
\]

Summing over the lines, we get \[q^k|c|X\Big(\dfrac{1}{s}+\dfrac{q}{X}+\dfrac{sq^k}{X^k}\Big)^{10^{-k}}.\] 

Lastly, Kane's original arguments show that
\[\sum_{ \text{lines }N\subset R}O(\min(Y^2|(b+v)^{-1}|,Y)) = O(X^{1 - 1/d}Y),
\] 
which is again within the desired bounds.
\end{proof}
Applying Abel summation and Lemma~\ref{proposition: Our 21} yields the following corollary.
\begin{corollary}
\label{corollary: Our 22}
Fix $K$ a number field of degree $d$ and $\xi$ a Gr\"{o}ssencharacter of modulus $\mathfrak{m}$.  Then given a positive number $X$ and an irrational number $\alpha$ which has a rational approximation of denominator $s$, we have that
\[
\Big|\sum_{\on{N}(\mathfrak{a})\le X}\log(\on{N}(\mathfrak{a}))\xi(\mathfrak{a})e(\alpha \on{N}(\mathfrak{a}))\chi(\on{N}(\mathfrak{a})) \Big| = O_{L, \xi}\Big(q^{d+1}X\log (X)\Big(\dfrac{1}{s}+\dfrac{1}{X^{1/d}}+\dfrac{s}{X}\Big)^{10^{-d}/2}\Big).
\]
\end{corollary}
Using the above lemmata, we can now attain our necessary estimate of $F_{K, \xi, \chi, X}(\alpha)$.

\begin{proof}[Proof of \cref{proposition: qdependentKanesum}] This proof adapts \cite[Proposition 23]{K13}, which in turn adapts Theorem 12.6 in \cite{IK04}, which is an application of Vaughan's identity.  We first see that
the necessary generalization of Equation (13.39) in \cite{IK04} still applies.  Letting $y=z=X^{2/5}$, we find that $F_{K, \xi, X,\chi}(\alpha)$ is the same as 

\begin{align}
    &\sum_{\substack{\on{N}(\mathfrak{ab})\leq X\\ \on{N}(\mathfrak{a})<X^{2/5}}} \mu(\mathfrak{a})\xi(\mathfrak{a})\log(\on{N}(\mathfrak{b})) e(\alpha \on{N} (\mathfrak{b})\on{N}(\mathfrak{a}))\chi(\on{N}(\mathfrak{a})) \\
    &- \sum_{\substack{\on{N}(\mathfrak{abc})\le X \\ \on{N}(\mathfrak{b}), \on{N}(\mathfrak{c})\le X^{2/5}}}\mu(\mathfrak{b})\Lambda_L(\mathfrak{c})\xi(\mathfrak{bc})\xi(\mathfrak{a})e(\alpha \on{N}(\mathfrak{bc})\on{N}(\mathfrak{a}))\chi(\on{N}(\mathfrak{abc})) \\
    &+ \sum_{\substack{\on{N}(\mathfrak{abc})\le X \\ \on{N}(\mathfrak{b}), \on{N}(\mathfrak{c})\ge X^{2/5}}}\mu(\mathfrak{b})\Lambda_L(\mathfrak{c})\xi(\mathfrak{ac})\xi(\mathfrak{b})e(\alpha \on{N}(\mathfrak{ac})\on{N}(\mathfrak{b}))\chi(\on{N}(\mathfrak{abc})) + O(X^{2/5}).
\end{align}

As in \cite{K13}, we bound the first time using Corollary~\ref{corollary: Our 22} on the sum over $\mathfrak{b}$.  Let $A=B^{1/4}\le X^{1/8}$. By Lemmata~\ref{lemma: Kane 18} and~\ref{lemma: Our 19}, we can bound the sum over terms where $\alpha \on{N}(\mathfrak{a})$ has no rational approximation with denominator between $A$ and $\frac{X}{A\on{N}(\mathfrak{a})}$ by 
\[
O\Big(X\Big(\log^2(X)(A^3B^{-1}+X^{-3/5}A^4)\Big(\dfrac{B}{A}\Big)^{\eps}\Big)\Big) = O(X\log^2(X)B^{-1/4+\eps}).
\]
For other values of $\mathfrak{b}$, Corollary~\ref{corollary: Our 22} bounds the sum as
\[
O(q^{d+1}X\log^2(X)(B^{-1/4}+X^{-3/5d})^{10^{-d}/2}).
\]

The second term is bounded similarly.  We let $A=\min(B^{1/4}, X^{1/41})$, and use Lemmas~\ref{lemma: Kane 18} and~\ref{lemma: Our 19} to bound the sum over terms with $\mathfrak{b}$ and $\mathfrak{c}$ such that $\on{N}(\mathfrak{bc})\alpha$ has no rational approximation with denominator $s$ between $A$ and $\frac{X}{A\on{N}(\mathfrak{bc})}$ by 
\[
O\Big(X\log^3(X)\Big(\dfrac{B}{A}\Big)^\eps (B^{-1/4}+X^{-1}A^4X^{4/5})\Big) = O(X\log^3(X)(B^{-1/4+\eps} + X^{-1/10})).
\]
Using \cref{proposition: Our 21}, we bound the sum over other values of $\mathfrak{b}$ and $\mathfrak{c}$ as 
\[
O(q^{d+1}X\log^2(X)(A^{-1}+X^{-1/5d})^{10^{-d}/2}).
\]

The final sum can be bounded exactly as in Proposition 23 of \cite{K13}.  The only difference is that we have factors of Dirichlet characters, but when they are inserted into the Cauchy-Schwarz argument the fact that their magnitude is $1$ implies that the same bounds hold.  So indeed, the final bound of
\[
O(X\log^2(X)(A^{-1}+X^{-2/5d})^{10^{-d}/4})
\]
holds.
Putting all of this together, we get the desired bound for $F_{K, \xi, X, \chi}(\alpha)$.

\end{proof}
\begin{proof}[Proof of \cref{thm: Gexponentialsumbound}]
This follows immediately by the orthogonality relations given by Equation \eqref{eq: DirichletOrthogonality}, \cref{prop: our Kane 6}, and \cref{proposition: qdependentKanesum}.
\end{proof}

\section{Bounded Gaps and Dense Clusters}\label{boundedGapsSect}

In \cite{M16}, Maynard extracts additional information about sets of well-distributed primes.  His results allow one to derive an asymptotic count for the number of primes with small gaps (i.e. the primes in dense clusters). 
To be precise, he stipulates the following. 
Given a set of integers $\cA$, a set of primes $\cP$, and $\mc{L} = \{L_1,\dots, L_k\}$ an admissible collection of linear functions $L(n)=l_1n+l_2$, we define
\begin{align}
    &\cA(x)=\{n\in\cA : x\le n < 2x\}, \qquad \cA(x; q, a) = \{n\in \cA(x), n\equiv a \mod q\}, \\ 
    &L(\cA) = \{L(n): n\in \cA\}, \qquad \varphi_L(q)=\varphi(|l_1|q)/\varphi(|l_1|), \\ 
    &\cP_{L, \cA}(x) = L(\cA(x))\cap \cP, \qquad \cP_{L, \cA}(x; q, a) = L(\cA(x; q, a))\cap \cP.
\end{align}

\begin{hypothesis}
\label{hypothesis: nice}
\cite{M16}
$(\cA, \cL, \cP, B, x, \theta)$.  Let $k = \#\cL$.
\begin{enumerate}
    \item $\cA$ is well distributed in arithmetic progressions: we have
    \[
    \sum_{q\le x^{\theta}}\max_{a}\Big|\#\cA(x; q, a)-\dfrac{\#\cA(x)}{q}\Big| \ll \dfrac{\#\cA(x)}{(\log x)^{100k^2}}.
    \]
    
    \item Primes in $L(\cA)\cap \cP$ are well distributed in most arithmetic progressions: for any $L\in\cL$ we have 
    \[
    \sum_{\substack{q\le x^{\theta} \\ (q, B) = 1}}\max_{(L(a), q)=1}\Big|\#\cP_{L, \cA}(x; q, a)-\dfrac{\#\cP_{L, \cA}(x)}{\varphi_L(q)}\Big| \ll \dfrac{\#\cP_{L, \cA}(x)}{(\log x)^{100k^2}}.
    \]
    
    \item $\cA$ is not too concentrated in any arithmetic progression: for any $q<x^{\theta}$ we have 
    \[
    \#\cA(x; q, a) \ll \dfrac{\#\cA(x)}{q}.
    \]
\end{enumerate}
\end{hypothesis}

Maynard proves the following result for sets satisfying Hypothesis~\ref{hypothesis: nice}.

\begin{theorem}
\label{thm: Maynard dense clusters}
[\cite{M16}, Theorem 3.1]
Let $\alpha > 0$ and $0<\theta < 1$.  Let $\mc{A}$ be a set of integers, $\mc{P}$ a set of primes, $\mc{L}=\{L_1, \ldots, L_k\}$ an admissible set of $k$ linear functions, and $B, x$ integers.  Let $L_i(n)=a_in+b_i\in\mc{L}$ have coefficients that satisfy $0\le a_i, b_i\le x^{\alpha}$ for all $1\le i\le k$, and let $k\le (\log x)^\alpha$ and $1\le B\le x^\alpha$.  

There is a constant $D_{\alpha,\theta}$ depending only on $\alpha$ and $\theta$ such that the following holds.  If $k\ge D_{\alpha,\theta}$ and $(\mc{A}, \mc{L}, \mc{P}, B, x, \theta)$ satisfy Hypothesis~\ref{hypothesis: nice}, and if $\delta>(\log k)^{-1}$ is such that
\[
\dfrac{1}{k}\dfrac{\varphi(B)}{B}\sum_{L\in\mc{L}}\dfrac{\varphi(a_i)}{a_i}\#\mc{P}_{L, \mc{A}}(x)\ge \delta\dfrac{\#\mc{A}(x)}{\log x},
\]
then 
\[
\# \{n\in \mc{A}(x): \#(\{L_1(n), \ldots, L_k(n)\}\cap \mc{P})\ge D_{\alpha,\theta}^{-1}\delta \log k\} \gg \dfrac{\#\mc{A}(x)}{(\log x)^k\exp(D_{\alpha,\theta}k)}.
\]

Moreover, if $\cP=\mathbb{P}, k\le (\log x)^{1/5}$ and all $L\in\mc{L}$ have the form $an+b_i$ with $|b_i|\le (\log x)k^{-2}$ and $a\ll 1$, then the primes counted above can be restricted to be consecutive, at the cost of replacing $\exp(Ck)$ with $\exp(Ck^5)$ in the bound.
\end{theorem}

Using Theorem~\ref{thm: Maynard dense clusters}, we now prove Theorem~\ref{thm: Bounded gaps}, which is a bounded gaps statement for $\mc{P}_{C, \alpha, \beta}$.  We do this by choosing our parameters to match the primes we are interested in, checking they satisfy the conditions, and translating the conclusion to match the desired statement. First we state a dense clusters theorem:

\begin{theorem}
Let $L/\QQ$ be a finite Galois extension with discriminant $\Delta_{L}$, and $\mc{B}_{n,\alpha,\beta} = \lfloor \alpha n + \beta \rfloor$ a Beatty sequence with $\alpha$ of finite type. Then there exists a constant $D_L$ only depending on the field $L$ such that if $m \in \NN$ and $k = \exp(D_L m)$, for any admissible set $\{n + b_1,\dots, n + b_k\}$ we have 
\[
\#\{x \le n \le 2x: \text{ at least $m$ of }n + b_i \in \mc{P}_{C,\alpha,\beta}\} \gg_m \frac{x}{(\log x)^{k}}.
\]
\end{theorem}

\begin{proof}
Let $\cA$ be the integers, let $\cL$ be an admissible set of the form $\{n+b_1, \ldots, n+b_k\}$, let $B$ be the discriminant $\Delta_L$ of the extension $L/\QQ$, and let $\cP = \cP_{C, \alpha, \beta}$ be the primes in a fixed Chebotarev class and Beatty sequence with $\alpha$ of finite type.
For the sake of convenience, we take $x$ large enough so that $b_i < x^{1/2}$ and $k \le (\log x)^{1/2}$, so that the conditions of \cref{thm: Maynard dense clusters} are satisfied.
We first check that these parameters do indeed satisfy Hypothesis~\ref{hypothesis: nice}.  $(1)$ and $(3)$ are trivially true for all the integers.  For $(2)$, by Theorem~\ref{thm: chebeattyBV} and the triangle inequality, we have 

\begin{align*}
    \sum_{\substack{q\le x^{\theta} \\ (q, B) = 1}}\max_{(L(a), q)=1}\Big|\#\cP_{L, \cA}(x; q, a)-\dfrac{\#\cP_{L, \cA}(x)}{\varphi_L(q)}\Big| 
    &\ll \sum_{\substack{q\le x^{\theta} \\ (q, B) = 1}}\max_{(L(a), q)=1} \Big|\pi_{C, \alpha, \beta}(2x+b_k-1; q, a) \\
    &-\dfrac{\pi_{C, \alpha, \beta}(2x+b_k-1)}{\varphi(q)}\Big| \\ 
    &\qquad+ \Big|\pi_{C, \alpha, \beta}(x+b_1-1; q, a)-\dfrac{\pi_{C, \alpha, \beta}(x+b_1-1)}{\varphi(q)}\Big| \\  
    &\ll \dfrac{\#\cP_{L, \cA}(x)}{(\log x)^{100k^2}}.
\end{align*}
For sufficiently large $x$, we have $\#\cP_{L, \cA}(x)>\dfrac{x}{2\log x}$ for all $L\in\mc{L}$ by the prime number theorem.  This implies that  
\[
\dfrac{1}{k}\dfrac{\varphi(B)}{B}\sum_{L\in\mc{L}}\dfrac{\varphi(a_i)}{a_i}\#\mc{P}_{L, \mc{A}}(x)\ge \dfrac{1}{k\Delta_L}\dfrac{kx}{2\log x}   \ge \dfrac{1}{2\Delta_L}\dfrac{\#\mc{A}(x)}{\log x}.
\]
Thus by Theorem~\ref{thm: Maynard dense clusters}, when $(2 \Delta_L)^{-1} > (\log x)^{-1}$ we have
\[
\# \{x\le n < 2x: \#(\{L_1(n), \ldots, L_k(n)\}\cap \mc{P}_{C, \alpha, \beta})\ge D^{-1}(2\Delta_L)^{-1}\log k\} \gg \dfrac{x}{(\log x)^k\exp(Dk)}.
\]
Setting $k \ge \lceil\exp(2mD\Delta_L)\rceil$, or equivalently $D_L \ge 2D\Delta_L$ yields the desired result.

\end{proof}

\begin{proof}[Proof of Theorem~\ref{thm: Bounded gaps}]

 Choose $b_j= p_{\pi(k)+j}$ where $j\in [1,k]$ and $p_j$ is the $j$th prime.  For $n\geq 6$,

$$n\log n + n\log \log n - n < p_n < n\log n + n \log \log n$$
If $n\geq 355991$, 
$$\frac{n}{\log n } \Big( 1+ \frac{1}{\log n }\Big) \leq \pi(n)\leq \frac{n}{\log n }\Big( 1 + \frac{1}{\log n}+\frac{2.51}{(\log n)^2}\Big)$$
Combining these two statements and 
doing an explicit Mathematica computation for small cases (see \cite{T14}), we have
$p_{\pi(k)+k}-p_{\pi(k)+1}\leq 1.6k \log k$
for all $k\geq 213$.  Then if $q_i$ is the $i$th prime in $\mc{P}_{C, \alpha, \beta}$, after adjusting $D$ so that $D \ge 1.6\cdot 213 \cdot \log(213)$, we get that
\[
\liminf_{n\rightarrow\infty}(q_{n+m}-q_n) \le mD\Delta_L\exp(2mD\Delta_L),
\]
as desired.

\end{proof}

\section{Arithmetic progressions in bounded gaps for Chebotarev intersect Beatty}\label{GTCB}
We can apply the formalism of prime gaps to prove \cref{corollary: GreenTaoMain}, a common generalization of the Green--Tao theorem and bounded gaps for primes in $\mc{P}_{C,\alpha,\beta}$. The theorem shows that we can find collections of arbitrarily long arithmetic progressions with bounded gaps.
This result will follow from a general criterion of Pintz, which marries the GPY method with Green--Tao arguments.

\subsection{The GPY setup}

Let $\mc{H} = \{h_1,\dots, h_k\}$ be an admissible set. Following Thorner \cite{T14} we let
\[
U = \prod_{\substack{p \le \log \log \log N\\p \nmid \Delta_L}}p
\]
for some positive integer $N$ to be chosen later.
This quantity differs from the $W$ in the Maynard--Tao arguments by the primes that divide $\Delta_L$, in order to exclude those moduli that ramify.
We first define 
\[
S_1(N,\cP_{C,\alpha,\beta}) \coloneqq \sum_{\substack{N\le n\le 2N \\ n\equiv u_0\pmod U}}
\Big(\sum_{d_i|n+h_i, \forall i}\lambda_{d_1, \ldots, d_k} \Big)^2,
\]
and
\[
S_2^{(m)}(N, \cP_{C, \alpha, \beta})\coloneqq \sum_{\substack{N\le n\le 2N \\ n\equiv u_0\pmod U}}
\chi_{\cP_{C,\alpha, \beta}}(n+h_m)
\Big(\sum_{d_i|n+h_i, \forall i}\lambda_{d_1, \ldots, d_k} \Big)^2.
\]
which sum to
\[
S_2(N, \cP_{\alpha, \beta}) = \sum_{i=1}^kS_2^{(i)}(N, \cP_{C,\alpha, \beta}).
\]
for weights $\lambda_{d_1,\dots, d_k}$ defined in terms of $U,F,$ and $R = N^{\theta/2 - \delta}$ where $\theta$ is the level of distribution.
The Maynard--Tao program seeks to show that 
\[
S_2(N,\cP_{C,\alpha,\beta}) - \rho S_1(N,\cP_{C,\alpha,\beta}) > 0
\]
for appropriate choices of weights $\lambda_{d_1,\dots, d_k}$ as well as $N, \rho$. This would then show at least $\rho$ primes in the interval $[N,2N)$ must be primes in a given set.
Maynard achieves such a result by finding formulas for $S_1,S_2$ as integral operators. The difference between our situation and Maynard's is that we exclude moduli that divide $\Delta_L$ in our Bombieri--Vinogradov estimate. Thorner precisely addresses this problem in case of Chebotarev classes.  We first define these integral operators and estimate $S_1$ and $S_2$.

\begin{definition}
\label{def:integral operators}
\cite{M15}
Let $F$ be a smooth function supported on $\cR_k = \{(x_1, \ldots, x_k)\in [0,1]^k : \sum_{i=1}^kx_i\le 1\}$.  Then define
\begin{align*}
    I_k(F) &\coloneqq \int_0^1\cdots\int_0^1 F(t_1, \ldots, t_k)^2dt_1\cdots dt_k, \\ 
    J_k^{(m)}(F) &\coloneqq \int_{0}^1\cdots \int_{0}^1\Big(\int_0^1 F(t_1, \ldots, t_k)dt_i\Big)^2dt_1\cdots dt_{i-1}dt_{i+1}\cdots dt_k.
\end{align*}
\end{definition}
The $S_1$ sum only has dependence on the excluded moduli, so our sum is the same as Thorner's:
\begin{proposition}[Proposition 4, \cite{T14}]
If $\cP$ has level of distribution $\theta > 0$ then 
\[
S_1 = (1 + o(1))\frac{\varphi(W)^kN(\log R)^k}{W^{k+1}}I_k(F),
\]
where $I_k^{(m)}(F)$ is defined as in Definition~\ref{def:integral operators}.
\end{proposition}
$S_2$ relies on the indicator function of $\cP_{C,\alpha,\beta}$, but the estimates only involve the prime number theorem and Bombieri--Vinogradov. Thus, the only term that differs from Thorner's $S_2$ is the density $\delta$, which is now $\alpha^{-1}|C|/|G|$.
\begin{proposition}
\label{prop:S2}
If the primes in $\cP_{C,\alpha, \beta}$ have level of distribution $\theta > 0$, then if $\delta=\frac{|C|}{\alpha|G|}$, we have 
\begin{equation}
    S_2(N, \cP_{C,\alpha, \beta}) = (1+o(1))\delta\varphi(\operatorname{rad}(\Delta_L))\dfrac{\log(R)}{\log(N)}\dfrac{\varphi(W)^kN\log(R)^k}{W^{k+1}}\sum_{i=1}^k J_k^{(i)}(F),
\end{equation}
where $J_k^{(m)}(F)$ is defined as in Definition~\ref{def:integral operators}.
\end{proposition}
\begin{proof}
The argument is essentially identical to the proof of Proposition 5 in Thorner~\cite{T14}, except we include the factor of $\alpha^{-1}$ to adjust for the fact that our primes come from a Beatty sequence. 
\end{proof}

\subsection{Pintz's argument}

\begin{theorem}[\cite{P10}, Theorem 5]\label{thm: GTPintz}
Let $\mc{H} = \{h_1,\dots, h_k\}$ be an admissible set. If there exists a set $S(\mc{H})$ and constants $c_1(k), c_2(k) > 0$ with 
\[
P^{-}\Big(\prod_{j = 1}^{k}(n + h_j)\Big) \ge n^{c_1(k)} \text{ for all }n \in S(\mc{H})
\]
and 
\[
\#\{n\le x : n \in S(\mc{H})\} \ge \frac{c_2(k)x}{(\log x)^k}
\]
for all $x$ sufficiently large, then $S(\mc{H})$ contains a $t$-term arithmetic progression for all $t \in \ZZ_{> 0}$.
\end{theorem}
To show that our collection of primes can satisfy these criteria, we need to adapt some lemmas from the work of Vatwani and Wong \cite{VW17}.
For corrections of some of the arguments in \cite{VW17}, see \cite{GKPPS19}.
\begin{proposition}\label{proposition: GTLem1}
Suppose we have an admissible set $\mathcal{H} = \{h_1,\dots, h_k\}$ with 
\[
k \gg \exp\Big(\frac{2\alpha |G|\Delta_Lm}{|C|\theta \varphi(\Delta_L)}\Big).
\]
Then there are infinitely many $n$ such that at least $m + 1$ of the $n + h_i$ lie in $\mc{P}_{C,\alpha,\beta}$ and 
\[
P^{-}\Big(\prod_{i = 1}^{k}(n+h_i)\Big) \ge n^{c_1(k)}.
\]
\end{proposition}
\begin{proof}
The proof follows the same reasoning as \cite[Theorem 5.4]{VW17}, with the only adjustment coming from the different density factor coming from the formula in \cref{prop:S2}.
\end{proof}

\begin{theorem}\label{thm: GTsuffmanyelts}
Assume the notation and hypotheses of \cref{proposition: GTLem1}. For sufficiently small chosen $c_1(k)$, let 
\[
S_{\mc{P}_{C,\alpha,\beta}}(\mathcal{H}) = \Big\{n \in \NN : \text{ at least } m + 1 \text{ of the $n+h_i$'s are in }\mc{P}_{C,\alpha,\beta}, \quad P^{-}\Big(\prod_{i = 1}^{k}(n + h_i)\Big) \ge n^{c_1(k)}\Big\}
\]
Then $\#\{n \le x: n \in S_{\mathcal{P}_{C,\alpha,\beta}}(\mc{H}) \ge c_{\mathcal{P}_{C,\alpha,\beta}}(k)x/(\log x)^k\}$ for some constant 
$
c_{\mathcal{P}_{C,\alpha,\beta}}
 > 0$.
\end{theorem}
\begin{proof}
The same proof as for \cite[Theorem 5.5]{VW17} applies, using in this case the slightly altered $S_2$ and \cref{proposition: GTLem1}.
\end{proof}
Using these two intermediate results alongside \cref{thm: GTPintz} we can prove \cref{corollary: GreenTaoMain}.
\begin{proof}[Proof of \cref{corollary: GreenTaoMain}]
We wish to apply the result of Pintz to the set
\[
S_{\mc{P}_{C,\alpha,\beta}}(\mathcal{H}') = \Big\{n \in \NN :  n + h_j' \in \mc{P}_{C,\alpha,\beta}\,\, \forall\, 1 \le j \le m+1, \, P^{-}\Big(\prod_{i = 1}^{k}(n + h_i)\Big) \ge n^{c_1(k)}\Big\},
\]
where $\mathcal{H}' \subseteq \mc{H}$ is some $(m+1)$-element subset. For Pintz's two criteria, the first criterion that the $n+h_i$ have large prime factors follows by definition. For the second criterion, first note  that there are only finitely many $(m+1)$-element subsets of $\mathcal{H}$. \cref{thm: GTsuffmanyelts} gives the desired result in the case that we do not specify which elements satisfy $n + h_i \in \cP_{C,\alpha,\beta}$.  The conclusion of \cref{thm: GTsuffmanyelts} must hold for some particular $(m+1)$-element subset $\mc{H}' \subseteq \mc{H}$ for a smaller constant $c_{\cP_{C,\alpha,\beta}}$, since we need to divide by the number of $(m+1)$-element subsets. Then applying Pintz's result, the corollary follows.
\end{proof}

\printbibliography
\end{document}